\documentclass[11pt,a4paper]{article}
\usepackage{graphicx, amsthm, amsmath, mathtools, amssymb, enumerate, ulem, comment, hyperref}
\usepackage[dvipsnames]{xcolor}

\title{Minimum blocking sets for families of partitions}
\author{Guillermo Gamboa Quintero\thanks{Computer Science Institute of Charles University, Prague, Czechia. \textit{E-mail}: gamboa@iuuk.mff.cuni.cz. Supported by Charles University project PRIMUS/24/SCI/012.}
\and
Ida Kantor\thanks{Computer Science Institute of Charles University, Prague, Czechia. \textit{E-mail}: ida@iuuk.mff.cuni.cz. Supported by GAČR grant 22-19073S.}}
\date{}

\begin{document}

\maketitle

\newtheorem{theorem}{Theorem}
\newtheorem{definition}{Definition}
\newtheorem{corollary}{Corollary}
\newtheorem{lemma}{Lemma}
\newtheorem{observation}{Observation}

\begin{abstract}
    A $3$-partition of an $n$-element set $V$ is a triple of pairwise disjoint nonempty subsets $X,Y,Z$ such that $V=X\cup Y\cup Z$. We determine the minimum size $\varphi_3(n)$ of a set $\mathcal{E}$ of triples such that for every $3$-partition $X,Y,Z$ of the set $\{1,\dots,n\}$, there is some $\{x,y,z\}\in \mathcal{E}$ with $x\in X$, $y\in Y$, and $z\in Z$. In particular,
        \[\varphi_3(n)=\left\lceil \frac{n(n-2)}{3} \right\rceil.\]

    For $d>3$, one may define an analogous number $\varphi_d(n)$. We determine the order of magnitude of $\varphi_d(n)$, and prove the following upper and lower bounds, for $d>3$:
    \[\frac{2 }{d!}\cdot n^{d-1} -o(n^{d-1}) \leq \varphi_d(n) \leq \frac{0.86}{(d-1)!}\cdot n^{d-1}+o(n^{d-1}).\]
\end{abstract}



\section{Minimum blocking set}
In this paper, we are interested in the following extremal problem: given $n$, find a smallest possible family $\mathcal{E}$ of triples of elements of $\{1,\dots,n\}$ such that whenever we partition $\{1,\dots,n\}$ into three nonempty sets $X,Y,Z$, there is some triple $\{x,y,z\}\in \mathcal{E}$ for which $x\in X$, $y\in Y$ and $z\in Z$. We call such family a {\em $3$-blocking set}.

Given a hypergraph $\mathcal{H}$, a {\em transversal} (or {\em edge cover}, or {\em blocking set}) is a set of vertices that intersects all hyperedges of $\mathcal{H}$. Transversals are objects that have been thoroughly studied in discrete mathematics (see \cite{alon90:transversals, bujtas12:transversals, henning18:transversals, henning20:transversals,lai90:transversals, lonc13:transversals} for a selected bibliography). Naturally, one is looking for a tranversal of minimum possible size. Our problem is a particular instance of this: consider the auxiliary hypergraph where the vertices are all triples of numbers in $\{1,\dots,n\}$, and each hyperedge corresponds to the partition $X,Y,Z$ of $\{1,\dots,n\}$ into three nonempty parts: the corresponding hyperedge contains all triples that have one vertex in each of the three parts.
For hypergraphs that are well-behaved, the asymptotic solution is known (see~\cite{GS} or \cite{Pippenger_Spencer}, and many later papers).
However, these results do not apply to our hypergraph, as it does not satisfy the conditions (it is far from regular, and the maximum codegree is close to the maximum degree).


Our problem also has a strong connection to Tur\'{a}n theory of hypergraphs. If $\mathcal{F}$ is a family of hypergraphs, then the Tur\'{a}n number ${\rm ex}(n,\mathcal{F})$ is the maximum number of edges in an $n$-vertex hypergraph that does not contain any hypergraph in $\mathcal{F}$ as a subgraph. Turán numbers of hypergraphs are notoriously hard to handle (see, for example, the nice surveys~\cite{Turan_type, Keevash}) but in a lot of particular cases, a progress has been made. In our case, if $B$ is a $3$-blocking set, its complement $B'$ (i.e., the family of triples that are not in~$B$) is a hypergraph that has no spanning complete $3$-partite subhypergraph. To see this, consider the case where $B'$ has such a subhypergraph with nonempty parts $X, Y, Z$. Then, it follows that $B$ does not contain a triple of the form $\{ x,y,z \}$ for $x \in X, y \in Y$ and $z \in Z$. So, we are looking for ${\rm ex}(n,\mathcal{F}_n)$, where $\mathcal{F}_n$ is the family of all spanning $3$-partite $3$-uniform hypergraphs on $n$ vertices. We have $\binom{n}{3} - {\rm ex}(n, \mathcal{F}_n)=\Theta(n^2)$, but we are interested in determining this number exactly. Usually in Tur\'{a}n theory problems, the family $\mathcal{F}$ is fixed and the same for all $n$, while in our case, the forbidden hypergraphs are spanning. Examples where the configuration in question is spanning include, e.g., theorems that specify conditions (in graphs or hypergraphs) that guarantee a Hamiltonian cycle, a packing by copies of a particular graph or the presence of a spanning subhypergraph of bounded degree (see, e.g.,~\cite{Kuhn_Osthus, Rodl_Ruc, alon_spanning}).

We may define a $d$-blocking set analogously, as follows.
\begin{definition}
If $V$ is a finite set and $d\in \mathbb{N}$, a {\em $d$-partition} of $V$ is a $d$-tuple of sets $\{X_1,\dots,X_d\}$ that are all nonempty and pairwise disjoint, and $V=X_1\cup\dots \cup X_n$.
\end{definition}

The set of all $d$-element subsets of $V$ will be denoted by $\binom{V}{d}$.
\begin{definition}
    Let $d$ be a positive integer. If $\mathcal{E}\subseteq \binom{V}{d}$, we say that $\mathcal{E}$ {\em blocks} the $d$-partition $\{X_1,\dots,X_d\}$ if there is a $d$-tuple $\{x_1,\dots,x_d\}\in \mathcal{E}$ such that $x_i\in X_i$ for all $i$. We call $\mathcal{E}$ a {\em $d$-blocking set} if it blocks all $d$-partitions of~$V$. Denote $\varphi_d(n)$ the size of the smallest $d$-blocking set for an $n$-element set $V$.
\end{definition}

In a way, a $d$-blocking set is a natural generalization of a connected graph to $d$-uniform hypergraphs. If $E$ is the edge set of a connected graph, then for every partition of vertices into two nonempty sets, we have an edge with one vertex in each set. In other words, the edge set of a connected graph is precisely what we call a $2$-blocking set. The minimum size of a connected graph on $n$ vertices is, trivially, $n-1$.

There are many ways to extend the notion of connectivity to $d$-uniform hypergraphs. In most of these, determining the minimum size of a connected hypergraph is not too interesting. However, we may, inspired by previous paragraph, say that a $d$-uniform  hypergraph is {\em arch-connected} if for every partition of the vertex set into $d$ nonempty classes, there is an edge intersecting all classes. Determining the minimum number of edges in a hypergraph which is connected in this sense is the same as determining the size of minimum $d$-blocking set. Some related problems seem to have been considered in the literature, but while we are looking for a partition that minimizes the number of edges (of a given hypergraph) that intersect all parts, in some papers the aim is to maximize this number, or to minimize the number of edges that are not contained in one part, etc. See, e.g.,~\cite{Conlon_Fox_Kwan_Sudakov}, ~\cite{Bollobas_Scott}, or~\cite{Papa_Markov}. See also~\cite{pttn_connected} for a related concept of {\em partition-connectedness}.


Most of this paper concentrates on $d=3$, which is the first non-trivial case of our problem. We determine the minimum size of a $3$-blocking set exactly, for all $n$.

\begin{theorem}\label{thm:main3}
For all $n$, we have $\varphi_3(n)=\left\lceil \frac{n(n-2)}{3} \right\rceil$.
\end{theorem}

For $d>3$, we determine the order of magnitude for the minimum size of a $d$-blocking set, and present nontrivial bounds on the multiplicative coefficient of the leading term.

\begin{theorem}
    For $d>3$, we have $\frac{2 }{d!}n^{d-1} -o(n^{d-1}) \leq \varphi_d(n) \leq \frac{0.86}{(d-1)!}n^{d-1}+o(n^{d-1}).$
\end{theorem}

The paper is organized as follows: In Section~\ref{sec:basictool} we prove a theorem which later proves useful for the upper bound for $d=3$, and use if immediately to get a lower bound. In Section~\ref{sec:construction3}, we describe a matching construction for the upper bound. The proof of the upper bound is a bit  technical, so we postopne it to Sections~\ref{sec:twoAi}--\ref{sec:2mod3}. In Sections~\ref{sec:lbhighd} and~\ref{sec:ubhighd} we prove a lower and an upper bound on $\varphi_d(n)$ for general $d\geq 4$, respectively.



\section{Basic tool and a lower bound for \texorpdfstring{$d=3$}{}}\label{sec:basictool}

For $d=1$, there is only one $1$-partition of $V$, with a single class equal to $V$. We are looking for a smallest family $\mathcal{E}$ of $1$-element sets such that for this $1$-partition, there is a set $\{x\}\in \mathcal{E}$ such that $x$ is in the single class of this partition. For any $x\in V$, the set $\{x\}$ does the job, so $\varphi_1(n)=1$.

For $d=2$, we want a set $\mathcal{E}$ of pairs of elements of $V$ such that for every partition of $V$ into disjoint sets $X$ and $Y$, there is some $\{x,y\}\in\mathcal{E}$ with $x\in X$ and $y\in Y$. That is, $\varphi_2(n)$ is the smallest number of edges in a connected graph on $n$ vertices, i.e., $\varphi_2(n)=n-1$.

To prove the exact value for $d=3$, the following will be useful.

\begin{definition}
    If $\mathcal{E}$ is a system of triples of elements of $V$ and if $X$ is a set, we define $G_{\mathcal{E}}(X)$ to be the graph with vertex set $V\setminus X$ and edges being all pairs $\{y,z\}$ such that $\{x,y,z\}\in \mathcal{E}$ for some $x\in X$.
\end{definition}

Note that there is no assumption that $X$ is a subset of $V$. In some proofs, this will be indeed the case---$X$ will not be a subset of $V$, but it will have a nonempty intersection with it.

\begin{theorem}\label{thm:characterization}
    $\mathcal{E}$ is a 3-blocking set if and only if for each $X\subseteq V$, the graph $G_{\mathcal{E}}(X)$ is connected.
\end{theorem}

\begin{proof}
 Suppose that $\mathcal{E}$ is a blocking set and $X\subseteq V$. Suppose for contradiction that $V\setminus X $ can be decomposed into disjoint nonempty subsets $Y,Z$ such that there are no edges of $G_{\mathcal{E}}(X)$ going from $Y$ to $Z$. Then $\{X,Y,Z\}$ is a 3-partition that is not blocked.

 Now suppose that for each $X\subseteq V$, the graph $G_{\mathcal{E}}(X)$ is connected, and let $\{X,Y,Z\}$ be a 3-partition. Since $G_{\mathcal{E}}(X)$ is connected, there is some $y\in Y$ and $z\in Z$ such that $\{y,z\}$ is an edge of $G_{\mathcal{E}}(X)$. That is, $\{x,y,z\}\in \mathcal{E}$ for some $x\in X$. Since $\{X,Y,Z\}$ was an arbitrary 3-partition, $\mathcal{E}$ is a blocking set.
\end{proof}

This theorem will be useful later in the proof of an upper bound for $\varphi_3(n)$, for showing that a certain set is blocking. But it is useful for finding a lower bound as well: if $\mathcal{E}$ is a blocking set, then $G_{\mathcal{E}}(\{x\})$ is connected for all $x$. In this case, for each $x\in V$, $G_{\mathcal{E}}(\{x\})$ has at least $n-2$ edges (where $n$ is the number of vertices), so at least $n-2$ triples contain $x$. By double counting the number of pairs $(x,e)$ where $e$ is a triple in $\mathcal{E}$ that contains $x$, we get
\[ \varphi_3(n)\geq \left\lceil \frac{n(n-2)}{3} \right\rceil\]

There are several constructions of blocking sets with $O(n^2)$ triples. It is easy to find, e.g., constructions with $\binom{n-1}{2}$ triples, but our goal is to match the lower bound exactly.

\section{Construction of a minimum blocking set for \texorpdfstring{$d=3$}{}}\label{sec:construction3}

In this section, we will use sets of $k$ numbered vertices, such as $\{u_0,\dots,u_{k-1}\}$, and their subsets. When indexing these vertices, for instance in writing $u_{2j+i-1}$ etc., all arithmetic operations are taken modulo $k$.

A set of vertices with consecutive indices, such as $\{u_j,u_{j+1},\dots,u_{j+s}\}$ for some $s$ will be called a {\em segment}. The indexing is again taken modulo $k$ (and the sequence $u_0,\dots,u_{k-1}$ understood as circular), so it may well happen that a segment has $u_{k-1}$ somewhere in the middle, followed by $u_0$.

Suppose now that $A$ and $B$ are disjoint $k$-element sets, with $A=\{a_0,\dots,a_{k-1}\}$ and $B=\{b_0,\dots,b_{k-1}\}$. Let us define a hypergraph $\mathcal{H}(A,B)$ on the vertex set $A\cup B$. Its edge set consists of all triples $\{a_i,a_j,b_{i+j}\}$ and all triples $\{a_i,a_j,b_{i+j+1}\}$.

For each $n\in \mathbb{N}$, we define an $n$-vertex set $V$ as follows.
For $k=\lfloor n/3\rfloor$, take the union of three disjoint $k$-element sets $A_0=\{u_0,\dots,u_{k-1}\}$, $A_1=\{v_0,\dots,v_{k-1}\}$, and $A_2=\{w_0,\dots,w_{k-1}\}$, and if $n$ is congruent to 1 or 2 (modulo 3), add also either the element $\infty$ or two elements $\infty_1,\infty_2$, respectively.

 Now let us define a set of triples $\mathcal{E}$ on the vertex set $V$. We start with all triples that are in the hypergraphs $\mathcal{H}(A_0,A_1)$, $\mathcal{H}(A_1,A_2)$, $\mathcal{H}(A_2,A_0)$, and add some more:

 \begin{itemize}
     \item If $n$ is divisible by 3, we add the $k$ triples $\{u_i,v_i,w_i\}$ for $i=0,\dots,k-1$.
     \item If $n \equiv 1 \; ({\mbox{mod }} 3)$, add the triples $\{\infty, u_i, v_i\}$, $\{\infty, v_i,w_i\}$, $\{\infty, w_i,u_{i+1}\}$ for all $i$.
     \item if $n \equiv 2 \; ({\mbox{mod }} 3)$, add the triples $\{\infty_1,u_i,v_i\}$, $\{\infty_1,v_i,w_i\}$, $\{\infty_2,w_i,u_{i+1}\}$, $\{\infty_2,u_i,v_{i+1}\}$, $\{\infty_1,\infty_2,w_i\}$ for all $i$.
\end{itemize}




For easier reference, let us call the triples in $\mathcal{H}(A_0,A_1)$, $\mathcal{H}(A_1,A_2)$, $\mathcal{H}(A_2,A_0)$ {\em red triples} and all other triples {\em blue triples}. Note that in all three cases, the total number of triples is equal to the lower bound that we obtained earlier, namely $\left\lceil \frac{n(n-2)}{3} \right\rceil$.

We now need to prove that in each case, the set $\mathcal{E}$ is a blocking set. That is, using Theorem~\ref{thm:characterization}, we need to verify that for each $X\subseteq V$, the graph $G_{\mathcal{E}}(X)$ is connected. Let us postpone the (somewhat long and technical) proof to Sections~\ref{sec:twoAi}--\ref{sec:2mod3} and let us first talk about a lower and an upper bound for higher dimensions.

\section{Lower bound for higher dimensions}\label{sec:lbhighd}

In this section, we prove an inductive lower bound on $\varphi_d(n)$.

\begin{definition}
    Let $(V,\mathcal{E})$ be a $d$-uniform hypergraph and $X$ a subset of $V$ such that $\lvert X \rvert \leq d-2$. We define $L_{\mathcal{E}}(X)$ to be the $(d-|X|)$-uniform hypergraph with vertex set $V \setminus X$ and hyperedges being all $\tau$ such that $X \cup \tau \in \mathcal{E}$.
\end{definition}
We may consider $\mathcal{E}$ to be an abstract simplicial complex (by including all subsets of the hyperedges). Viewed this way, the set $L_{\mathcal{E}}(X)$ is the {\em link} of the face $X$.


\begin{theorem}
\label{thm:lower_bound}
    For $d \geq 3$ and $n \geq d$, we have
    \[\varphi_d(n) \geq \frac{ 2 n (n-1) \cdots (n-(d-3)) (n-(d-1))}{d!}.\]
\end{theorem}
It follows from Theorem~\ref{thm:lower_bound} that if $d$ is fixed and $n\to \infty$, then
\[
    \varphi_d(n) \geq \frac{2 n^{d-1}}{d!} -o(n^{d-1}).
\]

\begin{proof}
    Let $\mathcal{E}$ be a blocking set. Let us count in two ways the number of pairs $(C,E)$, where $E\in \mathcal{E}$ and $C$ is a $(d-2)$-element subset of $E$.

    If $C=\{x_1,\dots,x_{d-2}\}$ is a $(d-2)$-element subset of  $[n]$ and $A\cup B$ is a partition of $[n]\setminus C$ into nonempty sets, then
    $\mathcal{E}$ blocks all partitions of the form $\{\{x_1\},\{x_2\},\dots,$ $\{x_{d-2}\},A,B\}$. This means that for each $(d-2)$-element set  $C$, if we take all sets in $\mathcal{E}$ that contain $C$ as a subset and restrict them to $[n]\setminus C$, we get a set of pairs which is a 2-blocking set of $[n]\setminus C$. So, $\mathcal{E}$ contains at least $(n-|C|)-1=n-d+1$ $d$-tuples that have $C$ as a subset, and the aforementioned number of pairs $(C,E)$ is at least $\binom{n}{d-2} \cdot (n-d+1)$.

    On the other hand, the number of $(d-2)$-tuples that each $d$-tuple in $\mathcal{E}$ contains is $\binom{d}{d-2}$. We get

    \[ |\mathcal{E}|\cdot \binom{d}{d-2}\geq \binom{n}{d-2}\cdot (n-d+1).
    \]

    Rearranging, we get the inequality in question.
\end{proof}

\section{Upper bound for higher dimensions}\label{sec:ubhighd}

If we select a point $x$ and take all $d$-tuples that contain $x$, this is a $d$-blocking set of size $\binom{n-1}{d-1}$, giving the trivial upper bound (for a fixed $d$ and for $n$ tending to $\infty$)
\begin{equation}\label{eq:trivub}
 \varphi_d(n)\leq \frac{n^{d-1}}{(d-1)!} -o(n^{d-1}).
\end{equation}

A theorem similar to Theorem~\ref{thm:characterization} characterizing the $d$-blocking sets in terms of $(d-1)$-blocking sets holds for all $d$. However, with increasing $d$, the situation seems to be too complex for such theorem to be useful. We will use a different approach and describe an inductive construction of a blocking set that establishes an upper bound better than the trivial one above.


Suppose that we have $n=2k$. Partition our point set into two $k$-elements sets, $A$ and $B$. For each $i\in \{0,\dots,d-1\}$, let $S_i=\binom{A}{i}$ and let $T_i$ be a $(d-i)$-blocking set of size $\varphi_{d-i}(k)$ for the set $B$ (i.e., consisting of $(d-i)$-element subsets of $B$). Define $W_i=\{s_i \cup t_i|s_i\in S_i \mbox{ and }t_i\in T_i\}$. Now let $W=\cup_{i=0}^{d-1} W_i$.

We claim that $W$ is a $d$-blocking set. To see that, let $X_1,\dots,X_d$ be a $d$-partition of $A\cup B$. Let $i$ be the number of indices $j$ such that $X_j\subseteq A$. If we take all $X_j$ such that $X_j\subseteq A$ and we pick an element from each of them, we will get an $i$-tuple, let us call it $s_i$. The remaining $d-i$ sets $X_j$, when restricted to $B$, form a $(d-i)$-partition of~$B$. So $T_i$ contains a $(d-i)$-tuple $t_i$ blocking this partition. The union $s_i\cup t_i$ blocks the partition $X_1,\dots,X_d$. But this set belongs to $W_i$, and thus to $W$. It follows that $W$ is a blocking set.

\begin{lemma}
    For all positive integers $d,k$, we have the following recurrence:
    \begin{align}\label{formula3}
\varphi_d(2k) &\leq \varphi_d(k) + \sum_{i=1}^{d-1}\binom{k}{i}\varphi_{d-i}(k). 
\end{align}
\end{lemma}

\begin{proof}
    Note that if $n<d$, we trivially have $\varphi_d(n)=0$. In the above construction, if $T_i$ has the smallest possible size, namely $\varphi_{d-i}(k)$, then $W_i$ has size $\binom{k}{i} \cdot \varphi_{d-i}(k)$. The size of the blocking set $W$ is an upper bound on $\varphi_d(2k)$, and $|W|=\sum_{i=0}^{d-1}=\varphi_d(k) + \sum_{i=1}^{d-1}\binom{k}{i}\varphi_{d-i}(k).$
\end{proof}

To deal with the case where $n$ is not even, we use the following recurrence:

\begin{lemma}
\label{recurrence1}
    For every $d$ and $n> d$, we have $\varphi_d(n)\leq \varphi_{d-1}(n-1) + \varphi_d(n-1)$.
\end{lemma}

\begin{proof}
    Select a point $x$ from the underlying set. Let $\mathcal{A}$ be a $(d-1)$-blocking set on the underlying set $[n]\setminus \{x\}$. If we add $x$ to each set in $\mathcal{A}$, the resulting family $\mathcal{A}'$ blocks all $d$-partitions where one of the parts is equal to $\{x\}$. Let $\mathcal{B}$ be a $d$-blocking set on the underlying set $[n]\setminus \{x\}$. This set also blocks all partitions of $[n]$ where $\{x\}$ is not a part of its own. Together, $\mathcal{A}'\cup \mathcal{B}$ is a $d$-blocking set of $[n]$ of size at most $\varphi_{d-1}(n-1) + \varphi_d(n-1)$.
\end{proof}


Now let us use the recurrence~(\ref{formula3}) to get an upper bound on $\varphi_d(n)$. We would like to find some constant $\gamma_d$ (dependent on $d$) such that, essentially, $\varphi_d(n)\leq \gamma_d n^{d-1}$. Since $\varphi_1(n)=1$, $\varphi_2(n)=n-1$ and $\varphi_3(n)=\lceil \frac{n(n-1)}{3}\rceil$, let us set $\gamma_1=\gamma_2=1$, and $\gamma_3=\frac{1}{3}$.

With a bit of foresight, for $d\geq 4$, we set
\[
 \gamma_d= \frac{1}{2^{d-1} -1} \cdot \sum_{i=1}^{d-1} \frac{\gamma_{d-i}}{i!}.
\]


\begin{theorem}
    For each $d$, there is a constant $\beta_d$ such that for every $n$ we have
    \[ \varphi_d(n) \leq \gamma_d \cdot n^{d-1} +\beta_d \cdot n^{d-2}.
    \]
\end{theorem}

\begin{proof}
    We use induction on $d$. For $d=1$, we have $\varphi_d(n)=1=\gamma_1\cdot n^0$, for $d=2$ we have $\varphi_2(n)=n-1=\gamma_2\cdot n^1$ and for $d=3$ Theorem \ref{thm:main3} 
    tells us that $\varphi_3(n) \leq \frac{n^2}{3}$, so $\beta_3=0$. Now let us suppose that the statement holds up to $d-1$ (for all $n$), and let us prove that it holds for $d$ (for all $n$). We want to estimate $\varphi_d(n)$ for a fixed $n$.

    If $n$ is odd, we use the recurrence from Lemma~\ref{recurrence1}. We get
    \begin{equation}\label{indbound}
        \varphi_d(n)\leq 
        \varphi_d\left(2\cdot \left\lfloor \frac{n}{2} \right\rfloor \right)+\varphi_{d-1}(n-1).
    \end{equation}
    If $n$ is even, then $\varphi_d(n)=\varphi_d\left(2\cdot \lfloor \frac{n}{2} \rfloor \right)$, so the bound~(\ref{indbound}) holds as well. Using the bound from~(\ref{formula3}) with $k= \lfloor \frac{n}{2} \rfloor$ and the  inductive hypothesis, we get

    \begin{align*}
        \varphi_d(n) &\leq \varphi_d \left( \left\lfloor \frac{n}{2} \right\rfloor \right)
            + \left(\sum_{i=1}^{d-1}\binom{ \left\lfloor \frac{n}{2} \right\rfloor  }{i} \cdot \varphi_{d-i}\left( \left\lfloor\frac{n}{2} \right\rfloor  \right)\right) +\varphi_{d-1}(n-1)\\
            &\leq \varphi_d \left( \left\lfloor \frac{n}{2} \right\rfloor  \right)
            + \left\lfloor \frac{n}{2} \right\rfloor ^{d-1} \cdot \left(\sum_{i=1}^{d-1} \frac{\gamma_{d-i}}{i!}\right)
            + \left\lfloor \frac{n}{2} \right\rfloor ^{d-2} \cdot \left(\sum_{i=1}^{d-1} \frac{\beta_{d-i}}{i!}\right) \\
            &+ n^{d-2}\cdot \gamma_{d-1} + n^{d-3}\cdot \beta_{d-1}\\
            &\leq \varphi_d \left( \left\lfloor \frac{n}{2} \right\rfloor \right)
            + {\frac{n}{2}}^{d-1} \cdot \left(\sum_{i=1}^{d-1} \frac{\gamma_{d-i}}{i!}\right) + n^{d-2}\cdot c_d
    \end{align*}
where $c_d$ is some constant that combines several other numbers, including the constants $\beta_1,\dots,\beta_{d-1}$ provided by the induction hypothesis.
    Rewriting the first term in an analogous way multiple times, we eventually get
    \begin{align*}
        \varphi_d(n) &\leq n^{d-1} \left(\sum_{i=1}^{d-1} \frac{\gamma_{d-i}}{i!}\right)\cdot \left(\frac{1}{2^{d-1}}+ \frac{1}{(2^{d-1})^2}+\frac{1}{(2^{d-1})^3}+\dots\right) \\
        &+ n^{d-2}\cdot c_d \cdot \left(1+\frac{1}{2^{d-2}}+ \frac{1}{(2^{d-2})^2}+\dots\right)\\
        &\leq \frac{n^{d-1}}{2^{d-1}-1} \sum_{i=1}^{d-1} \frac{\gamma_{d-i}}{i!}
        + n^{d-2}\cdot \frac{2^{d-2}}{2^{d-2}-1}\cdot c_d
    \end{align*}
The first term is equal to $n^{d-1}\cdot \gamma_d$, and the second is a constant multiple of $n^{d-2}$.
\end{proof}

Now we have an upper bound in terms of some number $\gamma_d$. This number is defined inductively, but it is not clear on the first sight how large it is. It is not hard to see, for instance using induction and the binomial theorem, that $\gamma_d\leq \frac{1}{(d-1)!}$, so our new bound is no worse than the trivial upper bound~(\ref{eq:trivub}). The following theorem shows that it is strictly better.

\begin{theorem}
    For all $d\geq 3$, we have $\gamma_d<\frac{0.86}{(d-1)!}$.
\end{theorem}

\begin{proof}
    We proceed by induction on $d$. For $d=3$, the inequality holds by definition of $\gamma_3=1/3$. Now suppose that we have some $d_0\geq 4$ and assume that the inequality is true for all $d$ with $3\leq d\leq d_0-1$. We have

    \begin{eqnarray*}
        &\gamma_{d_0}&= \frac{1}{2^{d_0 -1}}\left( \sum_{i=1}^{d_0-4} \frac{\gamma_{d-i}}{i!}+ \frac{\gamma_3}{(d-3)!} + \frac{\gamma_2}{(d-2)!} +\frac{\gamma_1}{(d-1)!}\right) \\
        &&\leq\frac{1}{2^{d_0 -1}}\left( \sum_{i=1}^{d_0} \frac{0.86}{i!(d_0-1-i)!}\right)+ \frac{(-29d_0^2+129d_0-58)}{2^{d_0 -1}\cdot(d_0-1)!\cdot 300}
    \end{eqnarray*}
    We use the binomial theorem to rewrite the first summand as $\frac{0.86}{(d_0 -1)!}$. If $d_0\geq 4$, the quadratic expression is negative, so the second summand may be omitted.
\end{proof}


In the above construction, the recursive process stops when we have either $1$, $2$, or $3$ points left, and at that point we use a known (optimal) construction for singletons, pairs, or triples. Some reflection reveals that if, instead of using our construction for triples, we take the recursive process one step further and stop only when we have $1$ or $2$ points left, we get a blocking set of size $\binom{n-1}{d-1}$. It is not too hard to see that this blocking set is larger than the one in our construction above (confirming in another way that $\gamma_d\leq \frac{1}{(d-1)!}$). So our improvement stems from better initialization. At the same time, this makes it clear that our construction is not optimal.


\section{The subgraphs induced by two of the \texorpdfstring{$A_i$}{}'s}\label{sec:twoAi}

Let us now go back to the proof of Theorem~\ref{thm:main3}. In Section~\ref{sec:construction3}, we described three constructions, according to the congruence class of $n$ modulo~3. As we mentioned, we will use Theorem~\ref{thm:characterization} to verify that in each of the three cases, for each $X\subseteq V$, the graph $G_{\mathcal{E}}(X)$ is connected.

This graph consists of some {\em red edges} that come from the red triples, and some {\em blue edges} that come from the blue triples (with some edges being possibly of both colors). 
In the next two sections, we analyze what the subgraph with the vertex set $(A_0\cup A_1\cup A_2)\setminus X$ and with the red edges looks like. We call this the {\em red subgraph}.

Let us go back to two $k$ elements sets $A=\{a_0,\dots,a_{k-1}\}$, $B=\{b_0,\dots,b_{k-1}\}$, and the hypergraph $\mathcal{H}(A,B)$.

\begin{lemma}\label{lemma:connectedA}
Let $X\subseteq A\cup B$ be a set that has a nonempty intersection with $B$.
\begin{enumerate}
    \item[$(i)$.] If $A \cap X =\emptyset$, then the subgraph of $G_{\mathcal{H}(A,B)}(X)$ induced by $A$ is connected. Also, in this case there are no edges of $G_{\mathcal{H}(A,B)}(X)$ going from $A$ to $B$.
    \item[$(ii)$.] If  $B\cap X=B$, then the subgraph of $G_{\mathcal{H}(A,B)}(X)$ induced by $A\setminus X$ is connected.
\end{enumerate}
That is, in both cases, the subgraph of $G_{\mathcal{H}(A,B)}(X)$ induced by $A\setminus X$ is connected.
\end{lemma}

\begin{proof}
 Let $b_i\in B\cap X$ and let us look at the edges of $G_{\mathcal{H}(A,B)}(X)$ that are there due to the triples that have two elements in $A$ and the third is this $b_i$. Since we have the triples $\{a_0,a_i,b_i\}$, $\{a_0,a_{i-1},b_i\}$, etc., we have two paths, $\{a_i,a_0,a_{i-1},a_1,a_{i-2},a_2,\dots\}$ and $\{a_i,a_{k-1},a_{i+1},\allowbreak a_{k-2},a_{i+2},a_{k-3},\dots\}$ that share the vertex $a_i$. If $X\cap A=\emptyset$, then the subgraph induced by $A$ is connected.

If $A\cap X\neq\emptyset$, then the vertices of $A$ that belong to $X$ could disconnect the path described above. In this case we use the fact that all vertices of $B$ belong to $X$. 
If $a_i$ and $a_j$ are two vertices of $A$ that are not in $X$, then since we have the triple $\{a_i,a_j,b_{i+j}\}$ and $b_{i+j}\in X$, $G_{\mathcal{H}(A,B)}(X)$ has the edge $\{a_i,a_j\}$, so the subgraph induced by $A\setminus X$ is complete.
\end{proof}

\begin{lemma}\label{lemma:connectedAB}
 Let $X\subseteq A\cup B$ and suppose that $A$ has some points that are in $X$, and $B$ has some points that are not in $X$. Then $G_{\mathcal{H}(A,B)}(X)$ is either connected or it has exactly two components. If $B\cap X=\emptyset$, then $G_{\mathcal{H}(A,B)}(X)$ is connected.

Suppose that $X\cap A=\{a_{j_1},a_{j_2},\dots\}$ for indices $j_1<j_2<\dots$. Suppose that it is the case that $G_{\mathcal{H}(A,B)}(X)$ has two components, say $I$ and $T$, and there is some $j_t$ such that $a_{j_t +1}\not \in X$ and $a_{j_t +1}\in I$. Then all of the following hold.
    \begin{itemize}
     \item For every $j_i$, if we take the vertex $a_{j_i}$ and the vertex $a_{j_{i+1}}$ that follows it in the (circular) sequence $a_{j_1},a_{j_2}\dots$ of vertices of $A\cap X$, the segment $\{a_{j_i +1},\dots,a_{j_{i+1} -1}\}$ (consisting of vertices that are not in $X$) has two non-empty parts: first a segment $\{a_{j_i +1},\dots \}$ consisting of vertices of $I$, followed by a segment $\{\dots,a_{j_{i+1}-1}\}$ consisting of vertices of $T$.
     \item Also, there is a vertex $v\in B$ that in $B$ immediately follows a vertex in $X$ and for which $v\in T$.
     \item And the component $I$ has some vertex of $B$. (It follows that both $I$ and $T$ have vertices of $B$.)
    \end{itemize}
\end{lemma}

Here $I$ and $T$ stand for ``initial'' and ``terminal''. Actually, more can be proved in addition to the statements in the lemma: if it is the case that we have two components, there is some divisor $m$ of $k$ such that whether a vertex belongs to $I$, $T$, or $X$ depends only on its class modulo $m$. That is, the pattern of colors $I$, $T$, $X$ in the sequence $u_0,u_1,\dots$ is very regular looking. But we will not use this in our proof.

\begin{proof}
 Suppose that, as in the statement of the lemma, $X\cap A=\{a_{j_1},a_{j_2},\dots\}$ for indices $j_1<j_2<\dots$. Let $W$ be the set of those vertices $a_{j_i +1}$ that do not belong to $X$. We will prove that all vertices of $W$ are in the same component. Let $a_{j_r +1},a_{j_s +1}$ be two vertices in $W$. Since $\{a_{j_r},a_{j_s +1},b_{j_r + j_s +2}\}\in \mathcal{H(A,B)}$ and $a_{j_r}\in X$, the pair $\{a_{j_s +1},b_{j_r +j_s +2}\}$ is an edge of $G_{\mathcal{H(A,B)}}(X)$. Similarly, since $\{a_{j_r +1},a_{j_s},b_{j_r + j_s +2}\}\in \mathcal{H(A,B)}$ and $a_{j_s}\in X$, the pair $\{a_{j_r +1},b_{j_r +j_s +2}\}$ is an edge. So if $b_{j_r +j_s +2}\not\in X$, then we have a path of length~2 from $a_{j_r +1}$ to $a_{j_s+1}$ through this vertex. If $b_{j_r +j_s +2}\in X$, then since $\{a_{j_r +1}, a_{j_s +1},b_{j_r +j_s +2}\}\in \mathcal{H}$, $\{a_{j_r +1},a_{j_s +1}\}$ is an edge of $G_{\mathcal{H(A,B)}}(X)$. In either case, $a_{j_r +1}$ and $a_{j_s +1}$ are in the same component. Let $I$ be this component.

Now suppose that $a_j$ and $a_{\ell}$ with $j<\ell$ are consecutive vertices in $A\cap X$ (that is, $a_{j+1},a_{j+2},\dots, a_{\ell -1}$ are not in $X$). 
If none of the vertices $b_{2j+2},b_{2j+3},\dots,b_{j+\ell -1}$ are in $X$, then we have a path $\{ a_{j+1},b_{2j+2},a_{j+2},\allowbreak b_{2j+3},\dots,b_{j+\ell -1},a_{\ell -1} \}$ connecting all these vertices to $a_{j+1}$, so they are all in $I$. This is because $a_j\in X$ and we have triples $\{a_j, a_{j+1},b_{2j+2}\}$, $\{a_j,a_{j+2},b_{2j+2}\}$, $\{a_j,a_{j+2},b_{2j+3}\}$, etc.

Now suppose that some of the vertices $b_{2j+2},b_{2j+3},\dots,b_{j+\ell -1}$ are in $X$, and this path is therefore cut into several shorter paths. If $b_{2j+t}\in X$ for some $t$, then we have the edge $\{a_{j+1},a_{j+t-1}\}$, since $\{a_{j+1},a_{j+t-1},b_{2j+t}\}\in \mathcal{H(A,B)}$. So for each of these shorter paths except for the very last one, the last of its vertices is in $I$ (and therefore all of its other vertices are, as well). The only shorter path between $a_{j+1}$ and $a_{\ell -1}$ that we haven't taken care of is the last one. And indeed, its vertices may be in a component different from $I$.

However, note that if $W'$ are all the vertices $a_{j_i -1}$ that are not in $X$, then all vertices in this set are in the same component. This may be proven completely analogously to the fact that the initial vertices $a_{j_i +1}$ are all in the same component (a few paragraphs above).

Now since
\begin{itemize}
 \item each segment of vertices between consecutive members of $X$ can be split into at most two subsegments (initial and terminal) with each of the two subsegments having vertices in the same component, and
 \item the set of initial vertices $W$ are all in the same component, and
 \item the set of terminal vertices $W'$ are all in the same component,
\end{itemize}
it follows that
\begin{itemize}
 \item either all vertices of $A$ that are not in $X$ are all in the same component,
 \item or each segment of vertices between consecutive members of $X$ can be split into two nonempty subsegments, the first having all vertices in the component $I$, and the second having all vertices in the component $T$.
\end{itemize}

Now let us prove the statement about the existence of a vertex in $B$ that immediately follows a vertex in $X$ and that is in $T$. Recall the setup from above, where $a_j$ and $a_{\ell}$ are two consecutive vertices in $A\cap X$. If none of $b_{2j+2},\dots,b_{j+\ell-1}$ are in $X$, then all vertices of $A\setminus X$ are in one component. The same conclusion holds if $b_{j+\ell} \in X$. So if we have two components, $I$ and $T$, then $b_{j+\ell}\not\in X$ and some of the vertices $b_{2j+2},\dots,b_{j+\ell -1}$ are in $X$. Let $b_{j+s}$ be the last such vertex. Then $a_s,b_{j+s+1},a_{s+1},b_{j+s+2},\dots,a_{\ell -1},b_{j+\ell}$ are all in $T$. In particular, $b_{j+s+1}$ is not in $X$ and it is in $T$, while following in $B$ immediately after the vertex $b_{j+s}$ that is in $X$.

We have shown that all vertices of $A\setminus X$ are in either one or two components of our graph. Now let us show that if $A$ contains some vertices that are not in $X$, then all vertices of $B\setminus X$ have edges to some vertices in $A\setminus X$. This will prove the first part of the lemma, namely that the graph as a whole has either one or two components.

Suppose for contradiction that there is a vertex $b_s$ in $B\setminus X$ that has no edges to $A\setminus X$. We know that $A\setminus X\neq \emptyset$, so let $j$ be such that $a_j\in X$. Now, since $a_j\in X$ and $\{a_j,a_{s-j-1},b_s\}\in \mathcal{H(A,B)}$, if $a_{s-j-1}\not\in X$, then $\{a_{s-j-1},b_s\}$ is an edge. It follows that $a_{s-j-1}\in X$. Similarly, since $\{a_{j+1},a_{s-j-1},b_s\}\in \mathcal{H(A,B)}$ and $a_{s-j-1}\in X$, we have $a_{j+1}\in X$, otherwise we would have the edge $\{a_{j+1},a_s\}$. But then also $a_{s-j-2}\in X$, etc, and eventually we conclude that all vertices $\{a_j,\dots,a_{s-j-1}\}$ are in $X$. Since also $a_{s-j}\in X$ (due to the triple $\{a_{j},a_{s-j},b_s\}$), we have $a_{j-1}\in X$ since $\{a_{j-1},a_{s-j},b_s\}\in \mathcal{H(A,B)}$. Then $a_{s-j+1}\in X$ as well, etc. We conclude that all vertices of $A$ belong to $X$, which is a contradiction with the assumption that $A\cap X\neq A$.

It remains to show that if our graph has two components, named $I$ and $T$ as above, then some vertex of $B$ belongs to $I$. Find $j$ such that $a_j\in I$ and $a_{j-1}\in X$. Let $t$ be an index such that $a_{j-1},a_{j-2},\dots,a_{j-t}\in X$ and $a_{j-t-1}\not\in X$. Then, by what we have shown above, $a_{j-t-1}\in T$. Since we have triple $\{a_{j-t},a_j,b_{2j-t}\}\in \mathcal{H}$ and $a_{j-t}\in X$, we have either $b_{2j-t}\not\in X$ and our graph has the edge $\{a_j,b_{2j-t}\}$, so $b_{2j-t}\in I\cap B$. Or $b_{2j-t}\in X$ and then $\{a_{j-t-1},a_j\}$ is an edge connecting the components $I$ and $T$, which is a contradiction.
\end{proof}

\section{The red subgraph}

The notation ($\mathcal{H}$, $A_i$ etc.) in the rest of the paper is the same as in the previous sections, in particular the section where we define the blocking set. When referring to the sets $A_0$, $A_1$, $A_2$ and writing, e.g., $A_{i+1}$, $A_{i+2}$, the indices of the sets are taken modulo~3.


\begin{lemma}\label{lemma:connected123}
Suppose that $X\subset V$ and suppose that none of $A_0$, $A_1$, $A_2$ are subsets of $X$. If, for some $i$, $A_i\cap X\neq \emptyset$ and $A_{i+1}\cap X\neq \emptyset$, then the red graph is connected.
\end{lemma}

\begin{proof}
 Suppose first that $A_{i+2}\cap X=\emptyset$. 
 By Lemma~\ref{lemma:connectedAB}, the subgraph of $G_{\mathcal{H}}(X)$ induced by $(A_{i+1}\cup A_{i+2})\setminus X$ is connected, so all of the vertices of $A_{i+1}\setminus X$ (of which there are more than zero) belong to the same component.
 The graph $G_{\mathcal{H}(A_i,A_{i+1})}(X)$ may have two components, but by Lemma~\ref{lemma:connectedAB}, each of them has vertices in $A_{i+1}\setminus X$. But since all vertices of $A_{i+1}\setminus X$ are in the same component of 
 the red graph, they provide connection between the two components.

 Now suppose that $A_{i+2}\cap X\neq \emptyset$. Suppose first that the subgraph induced by $(A_{i+1}\cup A_2)\setminus X$ is connected, so all vertices of $(A_{i+1}\cup A_{i+2})\setminus X$ are in the same component $C$. If the graph $G_{\mathcal{H}(A_i,A_{i+1})}(X)$ has two components, each of them has vertices in $A_{i+1}\setminus X$, so all of $(A_i\cup A_{i+1})\setminus X$ are in $C$ as well. In other words, all vertices of $(A_i\cup A_{i+1}\cup A_{i+2})\setminus X$ are in the same component, i.e., our graph is connected. The same argument shows that if the subgraph induced by $(A_i\cup A_{i+1})\setminus X$ or $(A_i\cup A_{i+2})\setminus X$  are connected, then the whole graph is.

 Now suppose that the graph $G_{\mathcal{H}}(X)$ has two components, say $C_1$ and $C_2$. Due to the previous paragraphs, this can only happen if each of the
 graphs $G_{\mathcal{H}(A_0,A_1)}(X)$, $G_{\mathcal{H}(A_1,A_2)}(X)$, $G_{\mathcal{H}(A_2,A_0)}(X)$ have two components.
 Find $j$ such that $u_j\in X\cap A_0$ and $u_{j+1}\not \in X$ and without loss of generality, suppose $u_{j+1}\in C_1$. Then, by Lemma~\ref{lemma:connectedAB}, whenever we have a segment of vertices in the circular sequence $u_0,u_1,\dots,u_{k-1}$ that do not belong to $X$, the first few belong to $C_1$ and the rest in $C_2$ (so $C_1$ is the same as $I$ and $C_2$ is the same as $T$ for $G_{\mathcal{H}(A_0,A_1)}(X)$). By Lemma~\ref{lemma:connectedAB}, we have a vertex in $A_1$ that does not belong to $X$, immediately follows a vertex in $X$, and belongs to $T$, that is, $C_2$. So for $G_{\mathcal{H}(A_1,A_2)}(X)$, the roles of $C_1$ and $C_2$ as $I$ or $T$ are switched: in each maximal segment of $A_1\setminus X$, the first few are in $C_2$ and the rest in $C_1$. And for the same reason, if we consider $A_1\cup A_2$, the labels are switched again in $A_2$: in each maximal segment of vertices of $A_2\setminus X$, we have first vertices of $C_1$, followed by vertices in $C_2$. But then, considering the last pair $A_2\cup A_0$, in in each segment of $A_0$ we have first vertices of $C_2$, followed by vertices of $C_1$. But this is not the case: we have started with the assumption that we have vertices of $C_1$ followed by vertices of $C_2$.
 \end{proof}

\begin{lemma}
\label{lemma:connected1}
    Let $X$ be a subset of $V$ such that one of the following holds:
    \begin{enumerate}
        \item Two of the sets $A_0, A_1, A_2$ are subsets of $X$.
        \item There is only one of $A_0, A_1, A_2$ that is a subset of $X$, say $A_i$, and $A_{i+1} \cap X \neq \emptyset$.
    \end{enumerate}
    Then the red graph is connected.
\end{lemma}

\begin{proof}
    If two of $A_0,A_1,A_2$ are subsets of $X$, then the red graph only contains the vertices of the third set and is connected by part (ii) of Lemma~\ref{lemma:connectedA}.

    Now suppose that the second set of assumptions is satisfied.  If $A_{i+2} \cap X=\emptyset$, then the subgraph induced by $A_{i+1}\cup A_{i+2}$ (which is actually the whole graph) is connected by Lemma~\ref{lemma:connectedAB} (with $A=A_{i+1}$ and $B=A_{i+2}$). If $A_{i+2} \cap X \neq \emptyset$, the graph $G_{\mathcal{H}(A_{i+1},A_{i+2})}(X)$ might have two components, but the edges that are there from $G_{\mathcal{H}(A_{i+2},A_{i})}(X)$ connect all the vertices of $A_{i+2}\setminus X$, so our graph is connected.
\end{proof}

\section{For \texorpdfstring{$n \equiv 0 \; ({\mbox{mod }} 3)$, $\mathcal{E}$}{} is a blocking set}

\begin{theorem}\label{thm:main3k}
    For $n \equiv 0 \; ({\mbox{mod }} 3)$, $\mathcal{E}$ is a blocking set.
\end{theorem}
\begin{proof}%
Let $X\subset V$. We will prove that $G_{\mathcal{E}}(X)$ is connected. If the hypotheses of Lemma~\ref{lemma:connected123} or Lemma~\ref{lemma:connected1} are satisfied, the red graph is connected, and since it is a subgraph of $G_{\mathcal{E}}(X)$ that has the same vertex set, $G_{\mathcal{E}}$ is also connected. If the hypotheses of neither of the two lemmas are satisfied, then the red graph alone is not enough. In this case, we need some blue edges, which are provided by Lemma~\ref{lemma:connected2} below.
\end{proof}

\begin{lemma}\label{lemma:connected2}
     Let $X$ be a subset of $V$ for which one of the following holds:
    \begin{enumerate}
        \item There is only one of $A_0, A_1, A_2$ that is a subset of $X$, say $A_i$, and $A_{i+1} \cap X = \emptyset$.
        \item None of the sets $A_0, A_1, A_2$ is a subset of $X$ and there is only one of $A_0$, $A_1$, $A_2$, say $A_i$, such that $A_i \cap X \neq \emptyset$.
    \end{enumerate}
Then, the graph $G_{\mathcal{E}} (X)$ is connected.
\end{lemma}

\begin{proof}
We separate the proof of the lemma according to the numbering on its statement.
\begin{enumerate}
    \item Without loss of generality, assume that $A_1 \cap X=A_1$ and also $A_2\cap X=\emptyset$. If $A_0 \cap X = \emptyset$, then the subgraph induced by $A_0$ is connected by Lemma~\ref{lemma:connectedAB} (i) with $A=A_0$ and $B=A_1$. The vertices of $A_2$ are connected to the vertices of $A_0$ by edges $\{u_i,w_i\}$ for all $i$ (since all vertices of $A_1$ are in $X$ and we have all triples $\{u_i,v_i,w_i\}$), so the whole graph is connected. If, on the other hand, $A_0\cap X\neq\emptyset$, then the subgraph induced by $A_2$ is connected by part (ii) of Lemma~\ref{lemma:connectedAB} (with $A=A_2$ and $B=A_0$). Each $u_i\in A_0\setminus X$ is connected to this subgraph by the edge $\{u_i,w_i\}$, similarly as above.
    \item If only one of the $A_i$ has vertices in $X$, say $A_1\cap X\neq \emptyset$, then the subgraph induced by $(A_1\cup A_2)\setminus X$ is connected by Lemma~\ref{lemma:connectedAB} and so is the subgraph induced by $A_0\setminus X$. To connect these two subgraphs, find a vertex $v_i$ of $A_1\cap X$. Then $\{u_i,w_i\}$ is an edge of $G_{\mathcal{H}}(X)$ since $v_i\in X$.
\end{enumerate}
\end{proof}


\section{For \texorpdfstring{$n \equiv 1 \; ({\mbox{mod }} 3)$, $\mathcal{E}$}{} is a blocking set}

\begin{lemma}
\label{lemma:inftycon}
    For every $X \subsetneq A_0 \cup A_1 \cup A_2$, there is a vertex $a \in (A_0 \cup A_1 \cup A_2) \setminus X$ such that $\{ \infty,  a \}$ is an edge of $G_{\mathcal{E}}(X)$.
\end{lemma}

\begin{proof}
Consider the circular sequence $\{ u_1, v_1, w_1, u_2, v_2, w_2, u_3, \dots, u_{k-1}, v_{k-1}, w_{k-1} \}$. Some, but not all, of its elements belong to $X$, so there is some $x\in X\cap (A_0\cup A_1 \cup A_2)$ such that the element $y$ immediately following $x$ in this sequence is not in $X$. Since the triple $\{\infty, x,y\}$ belongs to $\mathcal{E}$, the pair $\{\infty,y\}$ is an edge of $G_{\mathcal{E}}(X)$.
\end{proof}

We now prove an analogue of Lemma~\ref{lemma:connected2}.

\begin{lemma}
\label{lemma:2,3k+1}
    Let $X$ be a subset of $V$ such that one of the following holds:
        \begin{enumerate}
            \item There is only one of $A_0, A_1, A_2$ that is a subset of $X$, say $A_i$, and $A_{i+1} \cap X = \emptyset$.
            \item None of the sets $A_0, A_1, A_2$ is a subset of $X$ and there is only one of $A_0,$ $A_1,$ $A_2$, say $A_i$, such that $A_i \cap X \neq \emptyset$.
        \end{enumerate}
    Then, the graph $G_{\mathcal{E}} (X)$ is connected.
\end{lemma}

\begin{proof}
We separate the proof according to the numbering on the statement.
\begin{enumerate}
    \item Assume that $A_i \cap X=A_i$ and also $A_{i+1}\cap X=\emptyset$. First, consider the case where $\infty \in X$. If $A_{i+2} \cap X = \emptyset$, then the subgraph induced by $A_{i-1}$ is connected by Lemma~\ref{lemma:connectedAB} (i) with $A=A_{i-1}$ and $B=A_i$. Since $\infty \in X$ and none of the vertices of $A_{i+1}$ and $A_{i+2}$ are in $X$, there is a blue matching between the vertices of $A_{i+1}$ and $A_{i-1}$. So the whole graph is connected. If, on the other hand, $A_{i-1} \cap X \neq \emptyset$, then the subgraph induced by $A_{i+1}$ is connected by part (ii) of Lemma~\ref{lemma:connectedAB} (with $A=A_{i+1}$ and $B=A_{i-1}$). There is a blue matching between the vertices of $A_{i-1}\setminus X$ and the ``corresponding" vertices of $A_{i+1}$ (all of those are available, since $A_{i+1}$ has no vertices in $X$).

    Now, consider the case where $\infty \notin X$. Since $A_{i+1} \cap X = \emptyset$ and $A_i\cap X=A_i$, every vertex of $A_{i+1}$ is connected by a blue edge to $\infty$. For the analogous reasons, every vertex in $A_{i-1} \setminus X$ is also connected to $\infty$. The graph $G_{\mathcal{H}} (X)$ contains a spanning blue star, so it is connected.

    \item Assume that only $A_i$ satisfies $A_i \cap X \neq \emptyset$. Then, the subgraph induced by $(A_i\cup A_{i+1})\setminus X$ is connected and so is the subgraph induced by $A_{i-1}$. If $\infty \notin X$, then find a vertex $v_i$ of $A_i \cap X$. Then $\{u_i, \infty\}$ and $\{ \infty, w_i \}$ are edges of $G_{\mathcal{H}}(X)$. If $\infty \in X$, then we connect these subgraphs with using the edges $\{ u_{i+1}, w_i \}$.
\end{enumerate}
\end{proof}

\begin{theorem}\label{theorem:main3k+1}
    For $n \equiv 1 \; ({\mbox{mod }} 3)$, $\mathcal{E}$ is a blocking set.
\end{theorem}

\begin{proof}
Again, we prove that $G_{\mathcal{E}}(X)$ is connected, for each $X\subset V$. If $\infty$ is the only vertex not in $X$, then $G_{\mathcal{E}}(X)$ is trivially connected. Now suppose that this is not the case.

If the hypotheses of Lemma~\ref{lemma:connected123} or Lemma~\ref{lemma:connected1} are satisfied, the red subgraph is connected. If $\infty \in X$, the graph $G_{\mathcal{E}}(X)$ has the same vertices as the red subgraph, but possibly some more (blue) edges. It is therefore connected. If $\infty\not\in X$, then we have the (connected) red subgraph and an additional vertex $\infty$ and some blue edges going from the red subgraph to $\infty$. The existence of at least one such edge is guaranteed by  Lemma~\ref{lemma:inftycon}.

If the hypotheses of neither Lemma~\ref{lemma:connected123} nor Lemma~\ref{lemma:connected1} are satisfied, then $G_{\mathcal{E}}(X)$ is connected using Lemma~\ref{lemma:2,3k+1}.
\end{proof}

\section{For \texorpdfstring{$n \equiv 2 \; ({\mbox{mod }} 3)$, $\mathcal{E}$}{} is a blocking set}\label{sec:2mod3}

\begin{lemma}
\label{lemma:inftyonetwocon}
    Let $X \subsetneq A_0 \cup A_1 \cup A_2$. If $\infty_1$ (resp. $\infty_2$) is not adjacent to any point of $(A_0 \cup A_1 \cup A_2) \setminus X$ in $G_{\mathcal{H}}(X)$, then $\infty_2$ (resp. $\infty_1$) is adjacent to some point $a \in (A_0 \cup A_1 \cup A_2) \setminus X$ and $\{ \infty_1, \infty_2 \}$ is an edge of $G_{\mathcal{H}}(X)$.
\end{lemma}

\begin{proof}
    Let $X \subsetneq A_0 \cup A_1 \cup A_2$. If $\infty_1$ is not adjacent to any point of $(A_0 \cup A_1 \cup A_2) \setminus X$ in $G_{\mathcal{H}}(X)$ then for all $i \in \{ 0, 1, \ldots, k-1 \}$ it must hold that $\{ u_i, v_i, w_i \} \subset X$ or $\{ u_i, v_i, w_i \} \cap X = \emptyset$. In this case, since $X \neq A_0 \cup A_1 \cup A_2$, there is $i_0$ such that $\{ u_{i_0}, v_{i_0}, w_{i_0} \} \subset X$ and $\{ u_{i_0+1}, v_{i_0+1}, w_{i_0+1} \} \cap X = \emptyset$. Thus, we have that $\{ \infty_1, \infty_2 \}$ and $\{ \infty_2, u_{i_0+1} \} $ are edges of $G_{\mathcal{H}}(X)$ because of the triple $\{ \infty_1, \infty_2, w_{i_0} \}$ and $\{ \infty_2, w_{i_0} , u_{i_0+1} \}$, respectively. Now, if $\infty_2$ is not adjacent to any point of $(A_0 \cup A_1 \cup A_2) \setminus X$ in $G_{\mathcal{H}}(X)$ then for all $i \in \{ 0, 1, \ldots, k-1 \}$ it must hold that $\{ u_{i+1}, v_{i+2}, w_i \}$ $\subset X$ or $\{ u_{i+1}, v_{i+2}, w_i \} \cap X = \emptyset$. In this case, since $X \neq A_0 \cup A_1 \cup A_2$, there is $i_0$ such that $\{ u_{i_0 + 1}, v_{i_0 + 2}, w_{i_0} \} \subset X$ and $\{ u_{i_0+2}, v_{i_0+3}, w_{i_0+1} \} \cap X = \emptyset$. Thus, we have that $\{ \infty_1, \infty_2 \}$ and $\{ \infty_1, u_{i_0+2} \} $ is an edge of $G_{\mathcal{H}}(X)$ because of the triple $\{ \infty_1, \infty_2, w_{i_0} \}$ and $\{ \infty_1, v_{i_0+2} , u_{i_0+2} \}$, respectively. Note that at most one of $\infty_1, \infty_2$ is not adjacent to any vertex of $(A_0 \cup A_1 \cup A_2) \setminus X$ in $G_\mathcal{H}(X)$.
\end{proof}




\begin{lemma}
\label{lemma:2,3k+2}
    Let $X \subset A_0 \cup A_1 \cup A_2 \cup \{ \infty_1, \infty_2 \}$ such that one of the following holds:
        \begin{enumerate}
            \item There is only one of $A_0, A_1, A_2$ that is a subset of $X$, say $A_i$, and $A_{i+1} \cap X = \emptyset$.
            \item None of the sets $A_0, A_1, A_2$ is a subset of $X$ and there is only one of $A_0,$ $A_1,$ $ A_2$, say $A_i$, such that $A_i \cap X \neq \emptyset$.
        \end{enumerate}
    Then, the graph $G_{\mathcal{H}} (X)$ is connected.
\end{lemma}

\begin{proof}
We separate the proof according to the numbering on the statement.
\begin{enumerate}
\item If $\{\infty_1,\infty_2\}\subseteq X$, then $G_{\mathcal{H}}(X)$ has vertex set $(A_{i+1} \cup A_{i-1}) \setminus X$. By arguments similar to the proof of the first part of Lemma~\ref{lemma:connected2}, it is connected.

If $\infty_k\in X$ and $\infty_j\not\in X$ for $\{ j,k \} =  \{1,2\}$, then either all vertices of $(A_{i+1} \cup A_{i-1}) \setminus X$ are adjacent to $\infty_j$, or the subgraph of $G_{\mathcal{H}}(X)$ induced by $(A_{i+1} \cup A_{i-1}) \setminus X$ is connected by arguments similar to the proof of the first part of Lemma~\ref{lemma:connected2}, and there are some edges from $\infty_j$ to this subgraph.

If $\{\infty_1,\infty_2\}=\emptyset$, then there are $j,k \in  \{ 1,2 \}$ such that either $\infty_j$ is adjacent to all vertices in $(A_{i+1} \cup A_{i-1}) \setminus X$ and $\infty_k$ is adjacent to some other vertex by Lemma~\ref{lemma:inftyonetwocon}, or all points of $A_{i-1}\setminus X$ is adjacent to $\infty_j$ and all points of $A_{i+1}\setminus X$ are adjacent to $\infty_k$, and $\{\infty_1,\infty_2\}$ is an edge.


    \item By Lemma \ref{lemma:connectedA}, $A_{i-1}$ induces a connected subgraph and by Lemma \ref{lemma:connectedAB} so does $A_{i} \setminus X \cup A_{i+1}$. If $\{ \infty_1, \infty_2 \} \cap X \neq \emptyset$, there are blue edges connecting these two sets, due to $\infty_j \in X$ for some $j \in \{ 1,2 \}$. If $\infty_k \not \in X$ for $k=3-j$, then this vertex is adjacent to some other vertex. If none of $\infty_1,\infty_2$ is in $X$, then either one of them has edges to both sets ($A_{i-1}$ and $A_{i} \setminus X \cup A_{i+1}$) and the other is connected to the rest by Lemma~\ref{lemma:inftyonetwocon}, or one of $\infty_1, \infty_2$ is adjacent to vertices in $A_{i-1}$ and the other to vertices in $A_{i} \setminus X \cup A_{i+1}$ and $\{ \infty_1, \infty_2 \}$ is an edge.


\end{enumerate}

\end{proof}

\begin{theorem}\label{theorem:main3k+2}
    For $n \equiv 2 \; ({\mbox{mod }} 3)$, $\mathcal{E}$ is a blocking set.
\end{theorem}

\begin{proof}
    Again, we prove that $G_{\mathcal{E}}(X)$ is connected, for each $X \subset V$. If $\infty_1, \infty_2$ are the only vertices not in $X$, then $G_{\mathcal{E}}(X)$ is connected. Now suppose that this is not the case.

    If the hypotheses of Lemma~\ref{lemma:connected123} or Lemma~\ref{lemma:connected1} are satisfied, the red subgraph is connected. If $\infty_1, \infty_2 \in X$, the graph $G_{\mathcal{E}}(X)$ has the same vertices as the red subgraph, but possibly some more (blue) edges. It is therefore connected. If $\infty_j \in X$ and $\infty_k \notin X$, then we have the (connected) red subgraph and additional vertices $\infty_k$ and some blue edges going from the red subgraph to $\infty_k$. The existence of at least one such edge is guaranteed by  Lemma~\ref{lemma:inftyonetwocon}.

    Lastly, if the hypotheses of neither Lemma~\ref{lemma:connected123} nor Lemma~\ref{lemma:connected1} are satisfied, then $G_{\mathcal{E}}(X)$ is connected using Lemma~\ref{lemma:2,3k+2}.
\end{proof}

\section{Remarks and open problems}

This problem appeared naturally during our investigation of finite metric spaces, while trying to characterize the sets $\mathcal{E}$ of triples such that whenever we have a metric space where all triples in $\mathcal{E}$ are collinear (that is, the triples satisfy the triangle inequality with equality), then all points are collinear (any triple satisfies triangle inequality with equality).
See~\cite{anchors} for the results of this investigation, and~\cite{ChvSurvey} for more general context regarding finite metric spaces and lines in them.
However, it is clear that the problem is interesting in its own right.

For $d>0$, we were not able to find the asymptotic value of $\varphi_d(n)$. In this paper, we proved that there are constants $c_1$ and $c_2$ (dependent on $d$) such that $c_1\cdot n^{d-1}\leq \varphi_d(n) \leq c_2\cdot n^{d-1}$.

{\bf Problem 1.} Improve the constants $c_1$ or $c_2$.

Our feeling is that the lower bound might be tight, but that the upper bound is nowhere close to the truth. The ultimate goal is to close the gap between the constants.

{\bf Problem 2.} Does there exist a constant $c = c(d)$ such that $\lim_{n\to \infty} \frac{\varphi_d(n)}{n^{d-1}} = c$?

While we would like to know the value of such $c_d$, even just establishing the existence of this limit, without finding its value, would be interesting.

\bibliographystyle{plain}
\bibliography{ida}

\end{document}